\documentclass[]{article}

\usepackage{epstopdf}
\usepackage{amsfonts,amsbsy,amsmath,amsthm,amscd,amssymb,latexsym,cite,verbatim,texdraw,floatflt,
caption2,pb-diagram,graphicx}

\usepackage[numbers,sort&compress]{natbib}
\bibpunct[, ]{[}{]}{,}{n}{,}{,}

\theoremstyle{plain}
\newtheorem{theorem}{Theorem}[section]
\newtheorem{lemma}[theorem]{Lemma}

\theoremstyle{definition}
\newtheorem{definition}{Definition}

\theoremstyle{remark}

\begin{document}

\title{About convergence of solutions of one-dimensional stochastic equations}

\author{\textbf{Ivan H. Krykun}\thanks{This work was partially supported by a grant from the Simons Foundation (Award 1160640, Presidential Discretionary-Ukraine Support Grants, Ivan Krykun).} \textsuperscript{ a), b)}\\
\textsuperscript{a)} Institute of Applied Mathematics and Mechanics \\of NAS of Ukraine, Sloviansk, Ukraine\\
\textsuperscript{b)} KROK University, Kyiv, Ukraine\\
e-mail: iwanko@i.ua.}

\maketitle

\begin{abstract}
We consider a random process as a solution of stochastic differential equations with dependence of the coefficients on small parameter $\varepsilon$ and we suppose that the drift coefficients of these equations are unbounded on the parameter $\varepsilon$. We consider more general requirements on the convergence of some functions of coefficients of stochastic equations to limit functions. Necessary and sufficient conditions for the weak convergence of solutions of such stochastic equations if $\varepsilon$ tends to zero to a some stochastic equations involving a local time of process are obtained.
\end{abstract}


\section{Introduction}

It is well known (see \cite{StrVar79}, \cite{Makhno92}) that convergence of coefficients of stochastic differential equations is not sufficed for weak convergence of solutions of Ito's stochastic equation. It is necessary additional condition for diffusion coefficient of stochastic equation. For example, in \cite[theorem 11.3.3]{StrVar79} authors supposed that diffusion coefficient satisfies the Lipschitz condition.

G. Kulinich \cite{Kul83} and M. Portenko \cite{Por90} considered Ito's stochastic equation
$$x_\varepsilon(t)=x+\int_0^tb_\varepsilon(x_\varepsilon(s))ds+
\int_0^t\sigma_\varepsilon(x_\varepsilon(s))dw(s).$$

In \cite{Kul83} the necessary and sufficient conditions for weak convergence of some functions of solutions these equations to one-dimensional Ito's stochastic equations and for weak convergence of processes $x_\varepsilon$ to generalized diffusion processes, which were defined by M.~Portenko in \cite{Por76}, were obtained. In \cite[\S 3, Chapter III]{Por90} a weak convergence of function $b_\varepsilon (x)$ to Dirac $\delta-$function concentrated at zero in case of $\sigma_\varepsilon (x)\equiv 1$ and was proved weak convergence of processes $x_\varepsilon$ to generalized diffusion processes was considered. S.~Makhno in \cite{Makhno00}, \cite{Makhno04} proved requirements for convergence of measures generated by solutions of Ito's stochastic equations to the measure generated by solutions of some stochastic equations with local time.

In this paper we also consider a stochastic equation involving local time. It was firstly investigated by J. Harrison \& L. Shepp \cite{HarShe81} and J.-F. Le Gall \cite{Leg83}. In these works solution of stochastic process with local time was connected with skew Brownian motion, defined by K. Ito \& H. McKean \cite{ItoMck63}. Moreover in papers J.-F. Le Gall \cite{Leg83}, H.~Engelbert \& W. Schmidt \cite{EngSch91}, S. Makhno \cite{Makhno02} formulae which connect solutions of stochastic equations with local time with solutions of Ito's stochastic equations were obtained.

These problems and connected topics were discussed also in the works by S. Makhno \cite{Makhno05}, \cite[][Chapters 3-5]{Makhno12}, I. Krykun \cite{Krk05, Krk13-2, Krk17, Krk23} and I. Krykun \& S. Makhno \cite{KrkMah13}. Book by S. Makhno \cite{Makhno12}, as well as a book by J. Jacod \& A. Shiryaev \cite{JacShir87} also contain more detailed review of results of other authors.

In this article we consider Ito's stochastic equations with nonregular dependence of the coefficients on small parameter $\varepsilon$ and stochastic equations involving local time. Unlike the results of other authors we assume more general requirements on the convergence of some functions of coefficients of stochastic equations to their limit functions.

\section{Definition and notations}

We will use standard notation $(\Omega, \mathfrak{F}, \mathfrak{F}_{t}, \mathsf{P})$ for probability space with flow of $\sigma$-algebras $\mathfrak{F}_{t}, t\in [0, T]$ and $(w(t), \mathfrak{F}_{t})$ for the standard one-dimensional Wiener process with respect to the filtration $\mathfrak{F}_{t}$.

The indicator function of a set $A$ is denoted by $I_A(x)$.

Let us consider one-dimensional stochastic equation involving a local time
\begin{equation}\label{eqlch}
\xi(t)=x+\beta L^{\xi}(t,0)+\int_0^tb(\xi(s))ds+\int_0^t\sigma(\xi(s))dw(s),
\end{equation}
where $\beta$ be a constant and $b(x), \sigma(x)$ be a measurable functions.

\begin{definition}
According to \cite[definition 4.7 (1)]{EngSch91}, the equation (\ref{eqlch}) has
{\it a weak solution}, if for the given functions $b(x)$, $g(x)$ and
$\sigma(x)$ and the constant $\beta$ there is a probability space
$(\Omega, \mathfrak{F}, \mathfrak{F}_t, \mathsf{P})$ with flow of
$\sigma$-algebras $\mathfrak{F}_t, $ $t\geq0$, continuous semimartingale
$(\xi(t), \mathfrak{F}_t)$ and standard one-dimensional Wiener
process $(w(t), \mathfrak{F}_t)$ such that local time $L^\xi(t, 0)$ defined as
\begin{equation}\label{e3.1.3}
L^\xi(t, 0)=\lim_{\delta\rightarrow0}\frac{1}{2\delta}\int_0^tI_
{(-\delta,\delta)}(\xi(s))\sigma^2(\xi(s))ds
\end{equation}
exists almost surely and equation (\ref{eqlch}) fulfills almost surely.
\end{definition}

Let $(\mathbb{C}[0,T], C_t), t\in[0, T]$ be a space of continuous functions on interval $[0, T]$. Let us denote as $\mu_\varepsilon$ the measure on the functional space $(\mathbb{C}[0, T], C_t)$ generated by the process $\xi_\varepsilon$.

\begin{definition}
We study convergence of measures on functional space $(\mathbb{C}[0, T], C_t)$ generated by some random process. Such kind of convergence we will denote as $\Rightarrow$ and will name {\it a weak convergence of random process}.
\end{definition}

\begin{definition} \emph{The symmetric derivative of function} $h(x)$ we denote as $\mathbb{D}h(x)$ and it is defined as
$$\mathbb{D}h(x)=\lim_{\varepsilon\rightarrow0}\frac{h(x+\varepsilon)-h(x-\varepsilon)}{2\varepsilon}\;;$$
and \emph{the second distributional derivative of function} $h(x)$  we denote as $\mathbf{n}_h (dx)$ and it is defined for every infinitely differentiable function $H(x)$ with compact support by the equality
$$\int\frac{d^2 H(x)}{dx^2}h(x)dx=\int H(x)\mathbf{n}_{h}(dx).$$
\end{definition}

Let us denote
$$
sgnx=\left \{
\begin{aligned}
1,&\; \text{ for } \, x>0,\\
0,&\; \text{ for }\, x=0,\\
-1,& \; \text{ for } \, x<0.
\end{aligned}
\right.
$$

For the pair of measurable functions $(f(x), g(x))$ we will denote as
$$(f, g) \in \mathcal{L}(\lambda, \Lambda)$$
if there exist constants $\lambda, \Lambda$ such as $0<\lambda\leq\Lambda<\infty$ and the following inequalities
$$|f(x)|\leq\:\Lambda\:,\; \lambda\leq g(x)\leq\:\Lambda\:$$
are fulfilled for every $x$.


\section{Auxiliary results}

Let us consider a random process $X(t)$ as a solution of the following Ito's stochastic equation
\begin{equation}\label{itoeq}
X(t)=X(0)+\int_0^t{\alpha(X(s))}ds+\int_0^t{\gamma(X(s))}dw(s).
\end{equation}
We suppose that the pair of function $(\alpha, \gamma^2)\in \mathcal{L}(\lambda,\Lambda)$, then by \cite[theorem 4.35]{EngSch91} equation (\ref{itoeq}) has unique weak solution.

It follows from \cite[formula 4.3]{EngSch91}, that if for function $f(x)$ there exist derivatives $\mathbb{D}f(x)$ and $\mathbf{n}_f(dx)$,
then Tanaka's formula (generalized Ito's formula)
\begin{equation}\label{tanaka}
f(X(t))=f(X(0)) + \int^t_0 \mathbb{D}f(X(s))dX(s) +\frac{1}{2}\int L^X(t, y)\mathbf{n}_f(dy)
\end{equation}
is fulfilled.

Let functions $u_i(x),$ $i=1,2$ are some twice differentiable functions and let they have such properties:
$$u_1(0)=u_2(0)=0, u'_1(x)>0, \,\,u'_2(x)>0.$$
We define function $u(x)$ as follows
\begin{equation}\label{funcu}
u(x)=\left \{
\begin{aligned}
u_1(x),&  \text{ if }\, x\leq0,\\
u_2(x),&  \text{ if }\, x\geq0,
\end{aligned}\right.
\end{equation}
Further we denote $u_i'(0)=u_i$ and $v(x)=u^{-1}(x)$ be an inverse function of function $u(x)$. Then we obtain
$$\mathbb{D}u(x)=\frac{1}{2}(u'_2(x)+u'_1(x))+\frac{1}{2}(u'_2(x)-u'_1(x))sgn x ;$$
$$\mathbf{n}_u(dx)=(u_2-u_1)\delta_0(x)dx+\mathbb{A}_u(x)dx ;$$
where we define the function $\mathbb{A}_u(x)$ as follows
\begin{equation}\label{funcA}
\mathbb{A}_u(x)=\frac{1}{2}[(u''_2(x)+u''_1(x))+(u''_2(x)-u''_1(x))sgn x],
\end{equation}
and $\delta_0(x)$ is Dirac $\delta-$function concentrated at zero.

\begin{lemma}\label{l1}
{\it We assume that the equality $Y(t)=u(X(t))$ takes place. Then the following equalities
\begin{equation}\label{l1.1} L^Y(t,0)=\frac{u_1+u_2}{2}L^X(t,0), \end{equation}
\begin{equation}\label{l1.2} Y(t)=Y(0)+\frac{u_2-u_1}{u_2+u_1}L^Y(t,0)+\int_0^t{\widehat{\alpha}(Y(s))}ds+
\int_0^t{\widehat{\gamma}(Y(s))}dw,\end{equation}
are fulfilled, where function $\mathbb{A}_u(x)$ is defined by (\ref{funcA}), and we denote}
$$\widehat{\alpha}(x)=\mathbb{D}u(v(x))\alpha(v(x))+\frac{1}{2}\gamma^2(v(x))\mathbb{A}_u(v(x))\;,$$
$$\;\widehat{\gamma}(x)=\mathbb{D}u(v(x))\gamma(v(x)).$$

\end{lemma}
\begin{proof} For the function $u(x)$ and the process $X(t)$ by the Tanaka's formula (\ref{tanaka}) we obtain:
\begin{equation}\label{l1.3}
\begin{aligned} Y(t)=Y(0)+&\int_0^t\Big[\mathbb{D}u(X(s))\alpha(X(s))+
\frac{1}{2}\gamma^2(X(s))\mathbb{A}_u(X(s))\Big]ds+\\
+&\int_0^t\mathbb{D}u(X(s))\gamma(X(s))dw+\frac{u_2-u_1}{2}L^X(t,0).
\end{aligned}
\end{equation}

Since $X(s)=v(Y(s))$ then equality (\ref{l1.2}) follows from (\ref{l1.3}) and (\ref{l1.1}). Let us prove (\ref{l1.1}).
For this purpose we will apply the formula (\ref{tanaka}) to the process (\ref{itoeq}) and the function $\overline{u}(x)=|u(x)|$. We get :
$$\mathbb{D}\overline{u}(x)=\frac{1}{2}\big(u'_2(x)-u'_1(x)\big)+
\frac{1}{2}\big(u'_2(x)+u'_1(x)\big)sgn x ;$$
$$\mathbf{n}_{\overline{u}}(dx)=(u_2+u_1)\delta_0(x)dx+\widehat{A}_u(x)dx;$$
here we denote the function $\widehat{A}_u(x)$ as follows
$$\widehat{A}_u(x)=\frac{1}{2}\Big[\big(u''_2(x)-u''_1(x)\big)+\big(u''_2(x)+u''_1(x)\big)sgn x\Big].$$

Thus we get:
\begin{equation}\label{l1.4}
\begin{aligned}
\hskip 10 pt |Y(t)|=|Y(0)|+&\int\limits_0^t\Big[\mathbb{D}\overline{u}(X(s))\alpha(X(s))+\frac{1}{2}\gamma^2(X(s))
\widehat{A}_u(X(s))\Big]ds+\\
+&\int\limits_0^t{\mathbb{D}\overline{u}(X(s))\gamma(X(s))}dw+\frac{u_2+u_1}{2}L^X(t,0)\;.
\end{aligned}
\end{equation}
By formulae (\ref{tanaka}) and (\ref{l1.4}) we get
$$\begin{aligned}
L^Y(t,0)=&\frac{u_2+u_1}{2}L^X(t,0)+\int_0^t[\mathbb{D}\overline{u}(X(s))\alpha(X(s))+\\
+\frac{1}{2}\gamma^2(X(s))\widehat{A}_u(X(s))]ds+&\int_0^t{\mathbb{D}\overline{u}(X(s))\gamma(X(s))}dw-\int_0^t{sgn Y(s)}dY(s).
\end{aligned}
$$

The last integral on the right-hand side of this equation we derive from the formula (\ref{l1.3})
and we use such property of processes that $sgn X(s)=sgn Y(s)$ :
\begin{equation}\label{l1.5}
\aligned &\hskip 55 ptL^Y(t,0)=\frac{u_2+u_1}{2}L^X(t,0)+\\
&+\int_0^t\Big[\mathbb{D}\overline{u}(X(s))
\alpha(X(s))+\frac{1}{2}\gamma^2(X(s))\widehat{A}_u(X(s))\Big]ds+\\
&+\int_0^t{\mathbb{D}\overline{u}(X(s))\gamma(X(s))}dw+\int_0^t{\frac{u_2-u_1}
{2}sgn X(s)L^X(ds,0)}-\\
&-\int_0^t{sgn X(s)\Big[\mathbb{D}u(X(s))\alpha(X(s))+\frac{1}{2}\gamma^2(X(s))\mathbb{A}_u(X(s))\Big]}ds-\\
&-\int_0^t{sgn X(s)\mathbb{D}u(X(s))\gamma(X(s))}dw.
\endaligned
\end{equation}

Since the function $L^X(s, 0)$ increases only at the points $s$ such as $X(s)=0$ then
$$\int\limits_0^t{sgn X(s)L^X(ds,0)}=0.$$

Further we will group terms of the equation (\ref{l1.5}) and we will continue transformation of (\ref{l1.5}):
$$\begin{aligned}
&\int_0^t\Bigg[\mathbb{D}\overline{u}(X(s))\alpha(X(s))+\frac{1}{2}\gamma^2(X(s))\widehat{A}_u(X(s))-\\
&-sgn X(s)\Big[\mathbb{D}u(X(s))\alpha(X(s))+\frac{1}{2}\gamma^2(X(s))\mathbb{A}_u(X(s))\Big]\Bigg]ds+\\
&+\int_0^t\Big[\mathbb{D}\overline{u}(X(s))\gamma(X(s))-sgn X(s)\mathbb{D}u(X(s))\gamma(X(s))\Big]dw=
\end{aligned}$$
$$\begin{aligned}
=\int_0^t&\Bigg[\alpha(X(s))\bigg(\frac{u'_2(X(s))-u'_1(X(s))}{2}+\frac{u'_2(X(s))+u'_1(X(s))}{2}sgn X(s)-\\
&-\Big(\frac{u'_2(X(s))+u'_1(X(s))}{2}+\frac{u'_2(X(s))-u'_1(X(s))}{2}sgn X(s)\Big)sgn X(s)\bigg)+\\
&+\frac{1}{4}\gamma^2(X(s))\bigg\{\Big(\big(u''_2(X(s))-u''_1(X(s))\big)+\big(u''_2(X(s))+u''_1(X(s))\big)sgn X(s)\Big)-\\
&-\Big(\big(u''_2(X(s))+u''_1(X(s))\big)+\big(u''_2(X(s))-u''_1(X(s))\big)sgn X(s)\Big)\bigg\}sgn X(s)\Bigg]ds+\\
+\int_0^t&\gamma(X(s))\bigg[\frac{u'_2(X(s))-u'_1(X(s))}{2}+\frac{u'_2(X(s))+u'_1(X(s))}{2}sgn X(s)-\\
&-\Big(\frac{u'_2(X(s))+u'_1(X(s))}{2}+\frac{u'_2(X(s))-u'_1(X(s))}{2}sgn X(s)\Big)sgn X(s)\bigg]dw=
\end{aligned}$$
$$\begin{aligned}
&=\int_0^t\bigg(\alpha(X(s))I_{\{0\}}(X(s))\frac{u'_2(X(s))-u'_1(X(s))}{2}+\\
&\hskip 30 pt +\frac{1}{4}\gamma^2(X(s))I_{\{0\}}(X(s))\big(u''_2(X(s))-u''_1(X(s))\big)\bigg)ds+\\
&+\int_0^t{I_{\{0\}}(X(s))\frac{u'_2(X(s))-u'_1(X(s))}{2}\gamma(X(s))}dw=
\end{aligned}$$
$$\begin{aligned}
=\frac{1}{4}&\int_0^t{I_{\{0\}}(X(s))\gamma^2(X(s))(u''_2(X(s))-u''_1(X(s))}ds+\\
+&\int_0^t {I_{\{0\}}(X(s))\frac{u'_2(X(s))-u'_1(X(s))}{2}}dX(s)=
\end{aligned}$$
$$=\frac{1}{4}\gamma^2(0)\big(u''_2(0)-u''_1(0)\big)\int_0^t
{I_{\{0\}}(X(s))ds}+\frac{u_2-u_1}{2}\int_0^t{I_{\{0\}}(X(s))}dX(s)=0\:,$$
since $\displaystyle\int_0^t{I_{\{0\}}(X(s))}dX(s)=0$ (it follows from \cite[Lemma 5, p. 590]{GihSkor82}) and also since the following result:
$$\int_0^t{I_{\{0\}}(X(s))}dX(s)=\int_0^t{I_{\{0\}}(X(s))\frac
{\alpha(X(s))}{\gamma^2(X(s))}\gamma^2(X(s))}ds=$$
$$=\int_{\mathbb{R}}{I_{\{0\}}(y)\frac{\alpha(y)}{\gamma^2(y)}}L^X(t,y)dy=0$$
takes place. Here the first equality is a result from \cite[p. 217]{RevYor90} and the last one follows from \cite[Chapter VI, Corollary 1.6]{RevYor90}.
\end{proof}

\begin{definition} For some constant $|\beta|<1$ we define a function $\kappa(x)$ as follows
\begin{equation}\label{kappa}
\kappa(x)=\left\{
\begin{aligned}
(1-\beta)x, &\;\;\text{ if  } x<0 \\
(1+\beta)x, & \;\;\text{ if  } x\geq0
\end{aligned}\right.
\end{equation}
and a function $\varphi(x)=\left\{\aligned
\frac{x}{1-\beta}, \quad & \text{ if  }x<0 \\
\frac{x}{1+\beta},\quad & \text{ if  }x\geq0
\endaligned \right.$ is the inverse function of $\kappa(x)$. Furthermore, for every function $f$ we will denote a function $\widetilde{f}$ as follows
\begin{equation}\label{kappafunc}
\widetilde{f}(x)=\frac{f(\kappa(x))}{1+\beta sgn x}.
\end{equation}
\end{definition}
\medskip

Further we consider two following equations
\begin{equation}\label{eqxi}
\xi(t)=x+\beta L^\xi(t,0)+\int_0^t{g(\xi(s))}ds+
\int_0^t{\sigma(\xi(s))}dw(s),
\end{equation}
\begin{equation}\label{eqzeta}
\zeta(t)=\varphi(x)+\int_0^t{\widetilde{g}(\zeta(s))}ds+\int_0^t{\widetilde
{\sigma}(\zeta(s))}dw(s).
\end{equation}

For the measurable functions $g(x), \sigma(x)$ we introduce the following condition.

\medskip

{\bf Condition (I).}
\begin{itemize}
\item[$I_1.$] The absolute value of the constant at the local time is strictly less than 1.

\item[$I_2.$] The pair of the function $(g, \sigma^2)\in \mathcal{L}(\lambda, \Lambda)$.
\end{itemize}
\medskip

It's well known \cite[theorem 4.35]{EngSch91}, \cite{Ver79} that if Condition (I) holds then both the equation (\ref{eqxi}) and the equation (\ref{eqzeta}) have unique weak solutions. The next important lemma follows directly from the lemma \ref{l1}.

\begin{lemma}\label{l2}
{\it We assume that the coefficients of the process (\ref{eqxi}) satisfy the Condition (I). Then the process $\kappa(\zeta(t))$ is a solution of the equation (\ref{eqxi}).}
\end{lemma}

From the lemmas \ref{l1} and \ref{l2}, it follows that for the function $u$ defined by (\ref{funcu}) the following lemma is fulfilled.

\begin{lemma}\label{l3}
{ \it We assume that the coefficients of the process (\ref{eqxi}) satisfy Condition (I). Then the process
$\eta(t)=u(\xi(t))$ satisfies the following equation}
\begin{equation}\label{l3.1}
\eta(t)=u(x)+\frac{u_2-u_1+\beta(u_2+u_1)}{u_2+u_1+\beta(u_2-u_1)}L^\eta(t,0)+\int_0^t{g^*(\eta(s))}ds+\int_0^t\sigma^*(\eta(s))dw(s),
\end{equation}
{\it where we denote}
$$g^*(x)=\mathbb{D}u(u^{-1}(x))g(u^{-1}(x))+\frac{\sigma^2(u^{-1}(x))}{2}\mathbb{A}_u(u^{-1}(x)),$$
$$\sigma^*(x)=\mathbb{D}u(u^{-1}(x))\sigma(u^{-1}(x)).$$
\end{lemma}

\begin{proof}

We observe that equalities
$$\eta(t)=u(\xi(t))=u(\kappa(\zeta(t)))=\tau(\zeta(t))$$
take place, where function $\tau(x)$ is defined by equality
$$\tau(x)=u(\kappa(x))=\left\{
\begin{array}{cc}
u_1\big((1-\beta)x\big), &\text{ if  } x\leq0 \\
u_2\big((1+\beta)x\big), &\text{ if  } x\geq0
\end{array}.\right.$$
Thus we obtain the following results
$$\begin{aligned}
&\hskip 60 pt \tau_1(0)=\tau_2(0)=0\,;\;\\
&\tau'_1(x)=(1-\beta)u'_1\big((1-\beta)x\big)>0\,,\; \tau'_2(x)>0\,;\;\\
&\tau'_1(0)=\tau_1=(1-\beta)u_1\,,\; \tau'_2(0)=\tau_2=(1+\beta)u_2\,;
\end{aligned}$$
there exists an inverse function of function $\tau(x)$ that is defined as
$\tau^{-1}(x)=\varphi(u^{-1}(x))\,;$\\
and there exist second derivatives $\tau''_1(x)\,, \tau''_2(x)\,.$

Further we get
$$\begin{aligned}
&\mathbb{D}\tau(x)=(1+\beta sgn x)\mathbb{D}u(\kappa(x));\\
&\mathbb{A}_\tau(x)=(1+\beta sgn x)^2\mathbb{A}_u(\kappa(x));\\
\mathbf{n}_{\tau}(dx)=&(\tau_2-\tau_1)\delta_0(x)dx+(1+\beta sgn x)^2\mathbb{A}_u(\kappa(x))dx.
\end{aligned}$$

In such a way we conclude that the requirements of lemma \ref{l1} on equations (\ref{eqzeta}) and (\ref{l3.1}) with the function $\tau(x)$ are fulfilled. Therefore we obtain
$$\widehat{\alpha}(x)=\mathbb{D}\tau(\tau^{-1}(x))\widetilde{g}(\tau^{-1}(x))+\frac{1}{2}\widetilde{\sigma}^2(\tau^{-1}(x))
\mathbb{A}_\tau(\tau^{-1}(x))=$$
$$=(1+\beta sgn \tau^{-1}(x))\mathbb{D}u(\kappa(\tau^{-1}(x)))\frac{g(\kappa(\tau^{-1}(x)))}{1+\beta sgn \tau^{-1}(x)}+$$
$$+\frac{\sigma^2(\kappa(\tau^{-1}(x)))}{2(1+\beta sgn (\tau^{-1}(x)))^2}
(1+\beta sgn \tau^{-1}(x))^2\mathbb{A}_u(\kappa(\tau^{-1}(x)))=$$
$$=\mathbb{D}u(u^{-1}(x))g(u^{-1}(x))+\frac{\sigma^2(u^{-1}(x))}{2}\mathbb{A}_u(u^{-1}(x))=g^*(x);$$
$$\widehat{\gamma}(x)=\mathbb{D}\tau(\tau^{-1}(x))\widetilde{\sigma}(\tau^{-1}(x))=$$
$$=(1+\beta sgn \tau^{-1}(x))\mathbb{D}u(\kappa(\tau^{-1}(x)))\frac{\sigma(\kappa(\tau^{-1}(x)))}{1+\beta sgn \tau^{-1}(x)}=$$
$$=\mathbb{D}u(u^{-1}(x))\sigma(u^{-1}(x))=\sigma^*(x).$$
Finally by applying the lemma \ref{l1} we get statement of the lemma \ref{l3}.
\end{proof}

\section{Main result}

Let us consider the following Ito's stochastic differential equation
\begin{equation}\label{itoeqbg}
v_\varepsilon(t)=x+\int_0^t\big(b_\varepsilon(v_\varepsilon(s))+
g_\varepsilon(v_\varepsilon(s))\big)ds+\int_0^t{\sigma_\varepsilon
(v_\varepsilon(s))}dw(s).
\end{equation}

For measurable functions $g_\varepsilon(x), b_\varepsilon(x), \sigma_\varepsilon(x)$ we introduce the following condition.

\medskip
{\bf Condition (II).}
\begin{itemize}
\item[$II_1.$]
For every $\varepsilon>0$ there exists a unique weak solution of stochastic equation.

\item[$II_2.$] The pair of the function $(g_\varepsilon, \sigma_\varepsilon^2)\in \mathcal{L}(\lambda, \Lambda)$.

\item[$II_3.$] For every $x\in \mathbb{R}$
$$\left|\int_0^x{\frac{b_\varepsilon(y)}{\sigma^2_\varepsilon(y)}}dy\right|\leq \Lambda \,.$$
\end{itemize}
\medskip

We assume that coefficients of stochastic equation (\ref{itoeqbg}) satisfy Condition~(II). By $\nu_{\varepsilon}$ we denote measure on functional space $(\mathbb{C}[0,T], \mathcal{C}_T)$ generated by process $v_{\varepsilon}$.

Let us define functions $F_\varepsilon(x)$ and $f_\varepsilon(x)$ as follows
\begin{equation}\label{deffF}
F_{\varepsilon}(x)=\exp\Big\{-2\int\limits_0^x{\frac{b_{\varepsilon}
(y)}{a_{\varepsilon}(y)}dy}\Big\}, \; f_{\varepsilon}(x)=\int\limits_0^x{F_{\varepsilon}(y)}dy.
\end{equation}
Since $f_{\varepsilon}(x) $, as a function of $x$, is monotonically increasing, there exists an inverse function, which we will denote as $f_{\varepsilon}^{-1}(x)$.

Further we assume that
\begin{equation}\label{limoff}\lim\limits_{\varepsilon\rightarrow0}f_{\varepsilon}(x)=f(x)=\left
\{\begin{array}{cc}
f_1(x), & \text{if  }x\leq0 \\
f_2(x), & \text{if  }x\geq0
\end{array},\right.
\end{equation}
and we assume also that functions $f_1(x), f_2(x)$ are twice differentiable functions such as
 $f_1(0)=f_2(0)=0;$\, $f'_1(x)>0$ (for $x\leq0$) and $f'_2(x)>0$ (for $x\leq0$); and let's denote  $f_1=f'_1(0)\, ,  \,f_2=f'_2(0)$.

If the condition $II_3$ is fulfilled then the limit in (\ref{limoff}) exists uniformly on compact sets. Then $\displaystyle\lim_{\varepsilon\to 0}f_{\varepsilon}^{-1}(x)=f^{-1}(x)$ uniformly on compact sets and $f^{-1}(x)$ be an inverse function for function $f(x)$ defined in (\ref{limoff}).

We will prove below that the limit measure $\nu$ for measures $\nu_{\varepsilon}$ on the space $(\mathbb{C}[0, T], \mathcal{C}_T)$ is generated by such stochastic process involving local time
\begin{equation}\label{limproc}
v(t)=x+\alpha L^v(t,0)+\int_0^tg(v(s))ds+\int_0^t\sigma(v(s))dw(s).
\end{equation}

\medskip
\begin{theorem}\label{th1} {\it We assume that the coefficients of process (\ref{itoeqbg}) satisfy Condition (II), the coefficients of process (\ref{limproc}) satisfy Condition (I), and function $f(x)$ is defined in (\ref{limoff}). For weak convergence of measures generated by processes $\nu_{\varepsilon}\Rightarrow\nu$ if $\varepsilon\rightarrow0$, it's necessary and sufficient fulfilling of the following requirements:}
$$ \begin{aligned} &\texttt{\emph{a)}}\hskip 20 pt\displaystyle \alpha=\frac{f_1-f_2}{f_1+f_2};\\
&\texttt{\emph{aa)}}\hskip 15 pt \displaystyle \lim\limits_{\varepsilon\rightarrow0}\int\limits_0^x\frac{1}{F_
{\varepsilon}(y)a_{\varepsilon}(y)}dy=\int\limits_0^x\frac{1}
{a(y)\mathbb{D}f(y)}dy \;\;\; \text{\emph{for \; every}} \;\; x\in \mathbb{R} ;\\
&\texttt{\emph{aaa)}}\hskip 10 pt \displaystyle \lim\limits_{\varepsilon\rightarrow0}\int\limits_0^x\frac{g_
{\varepsilon}(y)}{a_{\varepsilon}(y)}dy=\int\limits_0^x\Big[\frac
{g(y)}{a(y)}+\frac{1}{2}\frac{\mathbb{A}_f(y)}{\mathbb{D}f(y)}\Big]dy \;\;\; \text{\emph{for \; every}} \;\; x\in \mathbb{R} .
\end{aligned}$$
\end{theorem}
\medskip

\begin{proof} Let's define the following process $\pi_{\varepsilon}(t)=f_{\varepsilon}(v_{\varepsilon}(t))$
and denote by $\varsigma_{\varepsilon}$ the measure generated by process $\pi_\varepsilon$ on the space $(\mathbb{C}[0,T], C_t)$. It follows from Ito's formula that for the process $\pi_{\varepsilon}(t)$ the following equation
\begin{equation}\label{e3.2.17}
\pi_{\varepsilon}(t)=\pi_{\varepsilon}(0)+\int_0^t{\widehat{g_{\varepsilon}}(\pi_{\varepsilon}(s))}ds+
\int_0^t{\widehat{\sigma_{\varepsilon}}(\pi_{\varepsilon}(s))}dw(s)
\end{equation}
is fulfilled. Here we denote
$$\widehat{g_{\varepsilon}}(x)=F_{\varepsilon}(f^{-1}_{\varepsilon}(x))g_{\varepsilon}(f^{-1}_{\varepsilon}(x))\:,\;
\widehat{\sigma_{\varepsilon}}(x)=F_{\varepsilon}(f^{-1}_{\varepsilon}(x))\sigma_{\varepsilon}(f^{-1}_{\varepsilon}(x)).$$

Let's define $\pi(t)=f(v(t))$ and denote by $\varsigma$ the measure generated by process $\pi$ on the functional space $(\mathbb{C}[0, T], C_t)$.

Further we will use \cite[lemma 5]{Makhno04}, so we need to verify that requirements \texttt{a)}, \texttt{aa)}, \texttt{aaa)} are necessary and sufficient conditions for convergence of measures $\varsigma_{\varepsilon}$ to the measure $\varsigma$. Due to the fact that the Condition (I) is fulfilled then the pair of functions $(\widehat{g_\varepsilon}, \widehat{\sigma_\varepsilon}^2)\in \mathcal{L}(a, A)$ for some constants $a, A$.

It follows from the lemma \ref{l3} that process $\pi(t)=f(v(t))$ satisfies equation (\ref{l3.1}) with such expression for the constant at the local time (one should substitute constants $u_1, u_2, \beta$ by constants
 $f_1, f_2, \alpha$ respectively):
$$\pi(t)=\pi(0)+\frac{f_2-f_1+\alpha(f_2+f_1)}{f_2+f_1+\alpha(f_2-f_1)}
L^\pi(t,0)+\int_0^t{\widehat{g}(\pi(s))}ds+\int_0^t\widehat{\sigma}(\pi(s))dw(s),$$
where we denote
$$\widehat{g}(x)=\mathbb{D}f(f^{-1}(x))g(f^{-1}(x))+\frac{\sigma^2(f^{-1}(x))}{2}\mathbb{A}_f(f^{-1}(x)),$$
$$\widehat{\sigma}(x)=\mathbb{D}f(f^{-1}(x))\sigma(f^{-1}(x)).$$

From another hand, it follows from the results by S. Makhno \cite{Makhno00} that the limit process for $\pi_{\varepsilon}(t)$ is a solution of Ito's stochastic equations, so constant at the local time of limit process has to be equal to zero. In such a way we can calculate a constant $\alpha$. Thus we conclude that requirement \texttt{a)} of the theorem has to be fulfilled.

Further from \cite[theorem 1]{Makhno00}, we conclude that the following requirements are necessary and sufficient conditions for convergence of measures $\varsigma_{\varepsilon}\Rightarrow\varsigma$ if $\varepsilon\rightarrow0$:
$$\int_0^x\frac{1}{\widehat{\sigma_{\varepsilon}}^2(y)}dy=\int_0^x\frac{1}{F^2_{\varepsilon}(f^{-1}_{\varepsilon}(y))\sigma^2_
{\varepsilon}(f^{-1}_{\varepsilon}(y))}dy=\int_0^{f^{-1}_{\varepsilon}(x)}\frac{1}{F_{\varepsilon}(z)\sigma^2_
{\varepsilon}(z)}dz\rightarrow$$
$$\rightarrow\int_0^{f^{-1}(x)}\frac{1}{\mathbb{D}f(z)\sigma^2(z)}dz=\int_0^x\frac{1}{\sigma^2(f^{-1}(y))\big[\mathbb{D}f(f^{-1}(y))\big]^2}dy=\int_0^x\frac{1}{\widehat{\sigma}^2(y)}dy\;;$$
$$\int_0^x\frac{\widehat{g_{\varepsilon}}(y)}{\widehat{\sigma_{\varepsilon}}^2(y)}dy=
\int_0^x\frac{g_{\varepsilon}(f_{\varepsilon}^{-1}(y))}{F_{\varepsilon}
(f^{-1}_{\varepsilon}(y))\sigma^2_{\varepsilon}(f^{-1}_{\varepsilon}(y))}dy=\int_0^{f_{\varepsilon}^{-1}(x)}\frac{g_{\varepsilon}
(z)}{\sigma^2_{\varepsilon}(z)}dz\rightarrow$$
$$\rightarrow\int_0^{f^{-1}(x)}\Big[\frac{g(z)}{\sigma^2(z)}+\frac{1}{2}\frac{\mathbb{A}_f(z)}{\mathbb{D}f(z)}\Big]dz=
\int_0^x\Big[\frac{g(f^{-1}(y))}{\sigma^2(f^{-1}(y))}+\frac{1}{2}\frac{\mathbb{A}_f(f^{-1}(y))}{\mathbb{D}f(f^{-1}(y))}\Big]\frac{dy}{\mathbb{D}f(f^{-1}(y))}=$$
$$=\int_0^x\Big[\frac{g(f^{-1}(y))}{\sigma^2(f^{-1}(y))\mathbb{D}f(f^{-1}(y))}+\frac{1}{2}\frac{\mathbb{A}_f(f^{-1}(y))}
{\big[\mathbb{D}f(f^{-1}(y))\big]^2}\Big]dy=\int_0^x\frac{\widehat{g}(y)}{\widehat{\sigma}^2(y)}dy\;,$$
that is equivalent to the requirements \texttt{aa)}, \texttt{aaa)} of the theorem.
\end{proof}

\section*{Acknowledgements}

The author expresses his heartfelt gratitude to the brave soldiers of the Ukrainian Armed Forces who protect the lives of the author and his family from Russian bloody murderers since 2014.

%

\end{document}